\documentclass[10pt]{amsart}
\usepackage[latin1]{inputenc}
\usepackage{amsmath}
\usepackage{amssymb}
\usepackage{latexsym}
\usepackage{verbatim}
\usepackage{graphics}
\usepackage{multirow}
\usepackage{amsfonts}
\usepackage{graphicx}
\usepackage{epstopdf}
\usepackage{pdflscape}
\usepackage{multicol}
\usepackage[margin=0.5in]{geometry}

\usepackage{blkarray}

\theoremstyle{plain}
\newtheorem{thm}{Theorem}
\newtheorem{quest}{Question}

\newtheorem{lem}{Lemma}
\newtheorem{cor}{Corollary}

\newtheorem{conj}{Conjecture}
\newtheorem*{conway}{Conway's Basic Theorem for Rational Tangles}

\newcommand{\abs}[1]{\left\vert #1 \right\vert}

\def\e{\varepsilon}

\def\s{\sigma}

\def\i{\scalebox{.15}{\raisebox{-.3\height}{\includegraphics{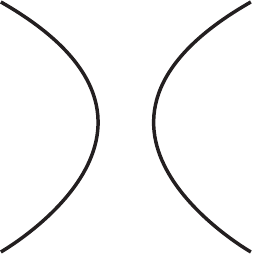}}}}
\def\z{\scalebox{.15}{\raisebox{-.3\height}{\includegraphics{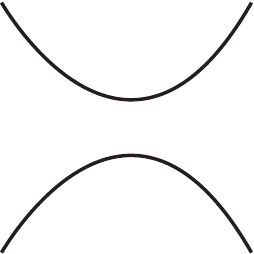}}}}
\def\v{\scalebox{.15}{\raisebox{-.3\height}{\includegraphics{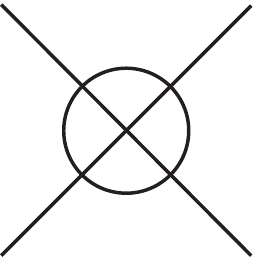}}}}

\theoremstyle{definition}

\theoremstyle{remark}

\begin{document}

\title{Virtual Rational Tangles}

\author{Blake Mellor}
\address{Loyola Marymount University}
\email{blake.mellor@lmu.edu}

\author{Sean Nevin}
\address{Loyola Marymount University}

\thanks{This research was partially supported by the Summer Undergraduate Research Program at LMU}

\begin{abstract}
We use Kauffman's bracket polynomial to define a complex-valued invariant of virtual rational tangles that generalizes the well-known fraction invariant for classical rational tangles.  We provide a recursive formula for computing the invariant, and use it to compute several examples.
\end{abstract}

\date{\today}

\maketitle

\section{Introduction}

A {\em tangle} is a proper embedding of a collection of arcs and circles in a 3-ball; by connecting the endpoints of the embedded arcs, the tangle can be {\em closed} to create a knot or link.  Tangles play an important role in the effort to classify knots and links \cite{co}.  Of particular interest are the {\em rational tangles}, whose closures are the {\em rational knots}, also known as {\em 2-bridge knots}.  These tangles, and knots, are amongst the simplest to construct; as a result, they appear in many applications, such as the study of DNA topology \cite{es}. They also have the advantage of being completely classified by the {\em fraction} of a rational tangle \cite{co, gk, kl}.

In this paper, we wish to consider rational tangles in the context of {\em virtual knot theory} \cite{ka2}.  We will review the definition of rational tangles, and extend them to {\em virtual rational tangles}.  We will then extend the fraction of a rational tangle to virtual rational tangles (in which context it is a rational complex number), prove it is an invariant, and find a recursive formula for computing it.  We conjecture that the extended fraction invariant may classify virtual rational tangles modulo the action of {\em flypes}.

\section{Rational Tangles} \label{S:tangles}

Formally, a {\em rational tangle} is a proper embedding of two arcs $\alpha_1, \alpha_2$ in a 3-ball $B^3$ (so the endpoints of the arcs are mapped to points on the boundary of the ball), such that there exists a homeomorphism of pairs:
$$h: (B^3, \{\alpha_1, \alpha_2\}) \rightarrow (D^2 \times I, \{x,y\} \times I)$$
In other words, a rational tangle can be turned into the trivial tangle simply by twisting the endpoints of the arcs around each other on the boundary of the ball (in particular, this means the two arcs are individually unknotted), along with isotopies inside the ball. So any rational tangle can be constructed (up to isotopy) from a trivial tangle of two horizontal or vertical strands by successively twisting pairs of neighboring endpoints to create crossings. We will describe this more precisely following the notation used by Kauffman and Lambropoulou \cite{kl}.

To begin with, we call the trivial horizontal and vertical tangles the $[0]$ tangle and $[\infty]$ tangle, respectively.  An $[n]$ tangle is the result of adding $n$ horizontal half-twists to the [0] tangle, while an $\frac{1}{[n]}$ tangle adds $n$ vertical half-twists to the $[\infty]$ tangle.  These elementary tangles are illustrated in Figure \ref{F:elementary}.

\begin{figure}[htbp]
\begin{center}
$$\scalebox{.8}{\includegraphics{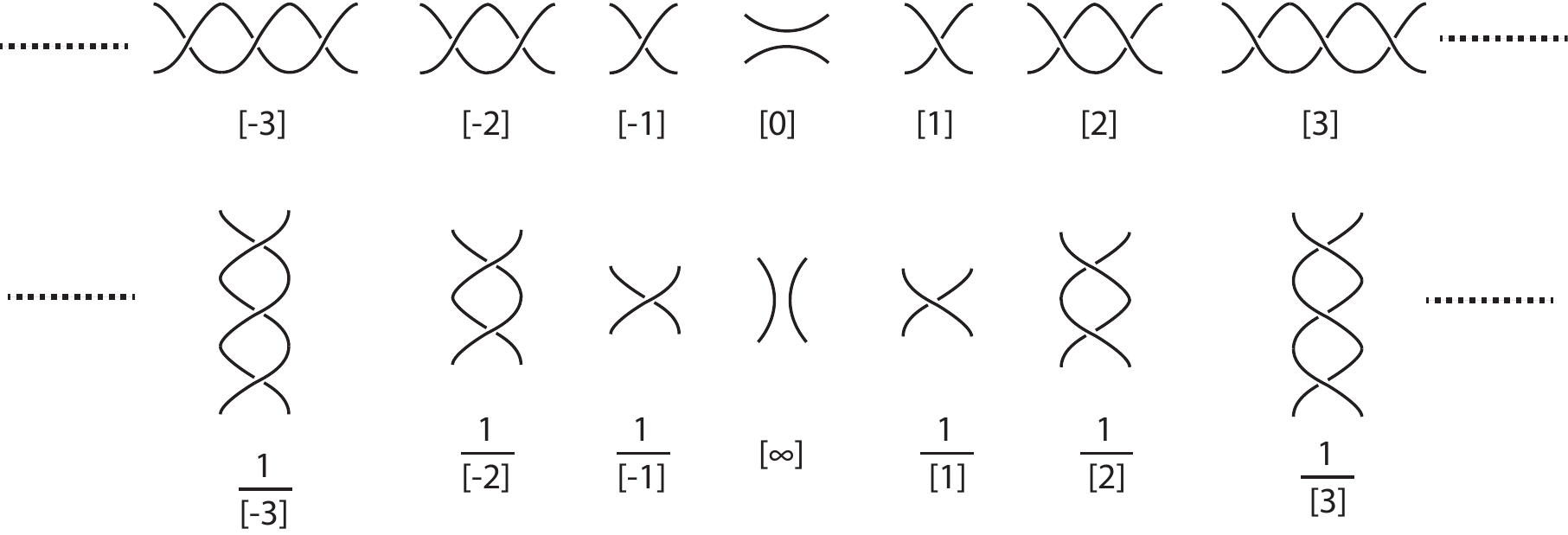}}$$
\end{center}
\caption{Elementary Tangles}
\label{F:elementary}
\end{figure}

Successively twisting the endpoints of a trivial tangle to construct a rational tangle can be broken down into a sequence of adding sets of horizontal and vertical half-twists.  To make this precise, we define two operations, + and $*$, on tangles with two strands.  These operations combine two tangles horizontally and vertically, respectively, as illustrated in Figure \ref{F:addtangles}.

\begin{figure}[htbp]
\begin{center}
$$\scalebox{1}{\includegraphics{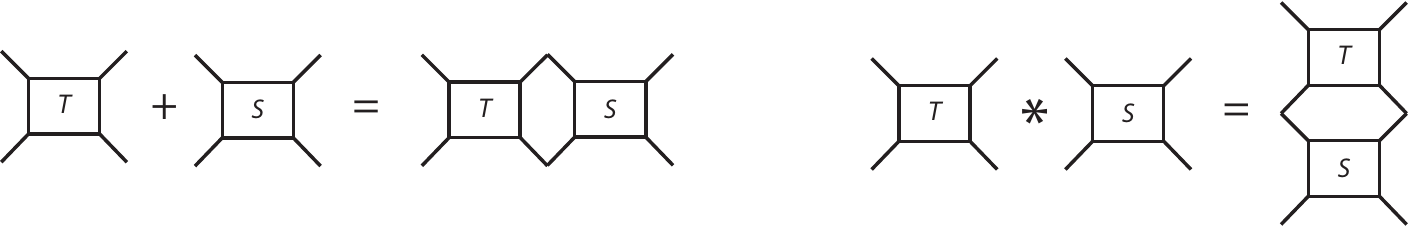}}$$
\end{center}
\caption{Adding Tangles}
\label{F:addtangles}
\end{figure} 

Goldman and Kauffman \cite{gk} show that, for any tangle $T$, $T + [n] \simeq [n] + T$ and $T * \frac{1}{[n]} \simeq \frac{1}{[n]} * T$ (where $\simeq$ denotes isotopy relative to the boundary of the ball).  Hence, up to isotopy, any rational tangle $T$ can be described inductively by setting $T_0 = [0]$ or $[\infty]$ and then alternating steps (1) and (2) below (where we start with (1) if $T_0 = [0]$ and with (2) if $T_0 = [\infty]$):
\begin{enumerate}
	\item $T_{k+1} = T_k + [n]$
	\item $T_{k+1} = T_k * \frac{1}{[n]}$
\end{enumerate}

So any rational tangle can be described by a vector $(a_1, a_2, \dots, a_n)$ (also called a {\em basic tangle}), with $a_i$ a non-zero integer, where either
$$(a_1, a_2, \dots, a_n) = \left([a_1]*\frac{1}{[a_2]}\right)+[a_3] \cdots \quad {\rm or} \quad (a_1, a_2, \dots, a_n) = \left(\frac{1}{[a_1]} + [a_2]\right)*\frac{1}{[a_3]} \cdots$$
Since, for any tangle $T$, $T + [0] \simeq T \simeq [\infty] * T$, we can unify these expressions by allowing $a_1 = \infty$ and $a_n = 0$.  Then each tangle has a unique vector $(a_1, a_2, \dots, a_n)$, with $n$ odd, where
$$(a_1, a_2, \dots, a_n) = \left(\cdots \left(\left([a_1]*\frac{1}{[a_2]}\right)+[a_3]\right)* \cdots * \frac{1}{[a_{n-1}]} \right) + [a_n]$$

The {\em fraction} of the rational tangle $(a_1, a_2, \dots, a_n)$ is the continued fraction
$$F(a_1, a_2, \dots, a_n) = [a_n, a_{n-1}, \dots, a_1] = a_n + \cfrac{1}{a_{n-1} + \cfrac{1}{\ddots\, + \cfrac{1}{a_2 + \cfrac{1}{a_1}}}}$$

The following theorem was first proved by Conway \cite{co}. Combinatorial proofs were given by Goldman and Kauffman \cite{gk} and Kauffman and Lambropoulou \cite{kl}.

\begin{conway}
Two rational tangles are isotopic if and only if they have the same fraction.
\end{conway}

\section{Kauffman bracket} \label{S:bracket}

The {\em bracket polynomial} $\langle K \rangle$ of a knot or link $K$ was introduced by Kauffman \cite{ka}.  It is defined by the following three rules, where $A$ is an indeterminate and $O$ is the unknot.  In rule (1), it is understood that the three knots are identical outside of the small region shown:

\begin{align}
\left\langle \scalebox{.15}{\raisebox{-.3\height}{\includegraphics{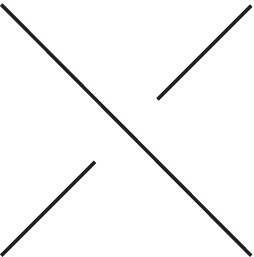}}} \right\rangle &= A\left\langle \scalebox{.15}{\raisebox{-.3\height}{\includegraphics{B2.pdf}}} \right\rangle + A^{-1}\left\langle \scalebox{.15}{\raisebox{-.3\height}{\includegraphics{B3.pdf}}} \right\rangle \\
\left\langle O \cup K \right\rangle &= (-A^2-A^{-2}) \langle K \rangle \\
\langle O \rangle &= 1
\end{align}

For a knot, the bracket polynomial is computed by applying rule (1) at every crossing, resulting in a sum of brackets of unlinks, with each  term multiplied by some power of $A$.  The polynomial for each unlink is computed using rules (2) and (3).  The bracket polynomial is an invariant of regular isotopy, which means it is invariant under Reidemeister II and Reidemeister III moves \cite{ka}.  However, it is not invariant under Reidemeister I moves.

$$\left\langle\scalebox{.3}{\raisebox{-.4\height}{\includegraphics{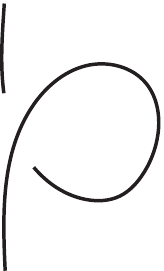}}}\right\rangle = -A^3 \left\langle \mid \right\rangle, \qquad \left\langle\scalebox{.3}{\raisebox{-.4\height}{\includegraphics{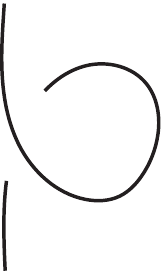}}}\right\rangle = -A^{-3} \left\langle \mid \right\rangle$$
\medskip

While the bracket polynomial was originally defined for knots, the definition works just as well for tangles.  However, after applying rule (1) to every crossing, each diagram will involve both unknotted circles and one of the two basic tangles: $\scalebox{.15}{\raisebox{-.3\height}{\includegraphics{B2.pdf}}}$ or $\scalebox{.15}{\raisebox{-.3\height}{\includegraphics{B3.pdf}}}$.  So the bracket of a tangle can be expressed as follows (now the diagrams on the right are the entire remaining tangle, not just a portion):

$$\left\langle T \right\rangle = f_T(A)\left\langle \scalebox{.15}{\raisebox{-.3\height}{\includegraphics{B3.pdf}}} \right\rangle + g_T(A)\left\langle \scalebox{.15}{\raisebox{-.3\height}{\includegraphics{B2.pdf}}} \right\rangle$$

\medskip

The coefficients $f_T(A)$ and $g_T(A)$ are invariant under Reidemeister moves II and III, but are changed by Reidemeister move I by the same factor as the overall bracket.  However, this means that the {\em ratio} of these coefficients is invariant under Reidemeister I.  We define $R_T(A) = f_T(A)/g_T(A)$; this is an isotopy invariant of tangles.

Goldman and Kauffman \cite{gk} showed that the {\em conductance invariant} $C(T) = iR_T(\sqrt{i})$ is always equal to the fraction of the tangle (thereby showing that the fraction is an isotopy invariant of tangles).  This was a key part of their proof of Conway's Basic Theorem (the other part was showing that if two tangles had the same fraction, then they were isotopic).

\section{Virtual Rational Tangles} \label{S:vrt}

Virtual knots were introduced by Kauffman \cite{ka2} as a generalization of classical knot theory (one motivation is to find a better correspondence between knots and Gauss diagrams; another is to represent knots in thickened surfaces).  Kauffman showed that virtual knots can be defined as equivalence classes of diagrams modulo certain moves, generalizing the Reidemeister moves of classical knot theory. Diagrams for virtual knots contain both classical crossings (positive and/or negative crossings, if the knot is oriented) and \emph{virtual} crossings, as shown in Figure \ref{F:crossings}.  Two diagrams are equivalent if they are related by a sequence of the Reidemeister moves shown in Figure \ref{F:reidemeister}. Note that moves (I)--(III) are the classical Reidemeister moves. Kauffman \cite{ka2} showed that classical knots are equivalent by this expanded set of Reidemeister moves if and only if they are equivalent by the classical Reidemeister moves, so classical knot theory embeds inside virtual knot theory.

\begin{figure}[htbp]
\begin{center}
\scalebox{.6}{\includegraphics{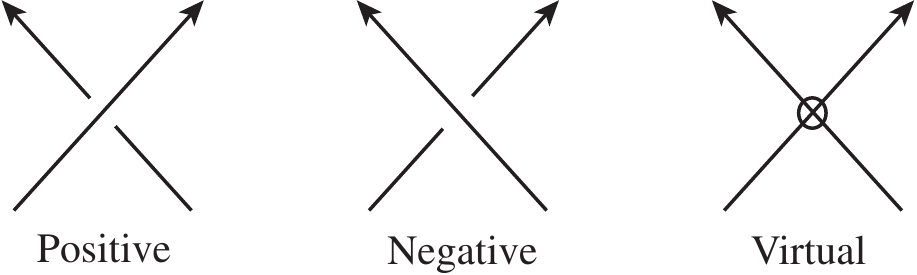}}
\end{center}
\caption{Classical and virtual crossings}
\label{F:crossings}
\end{figure}

\begin{figure}[htbp]
\begin{center}
\scalebox{.8}{\includegraphics{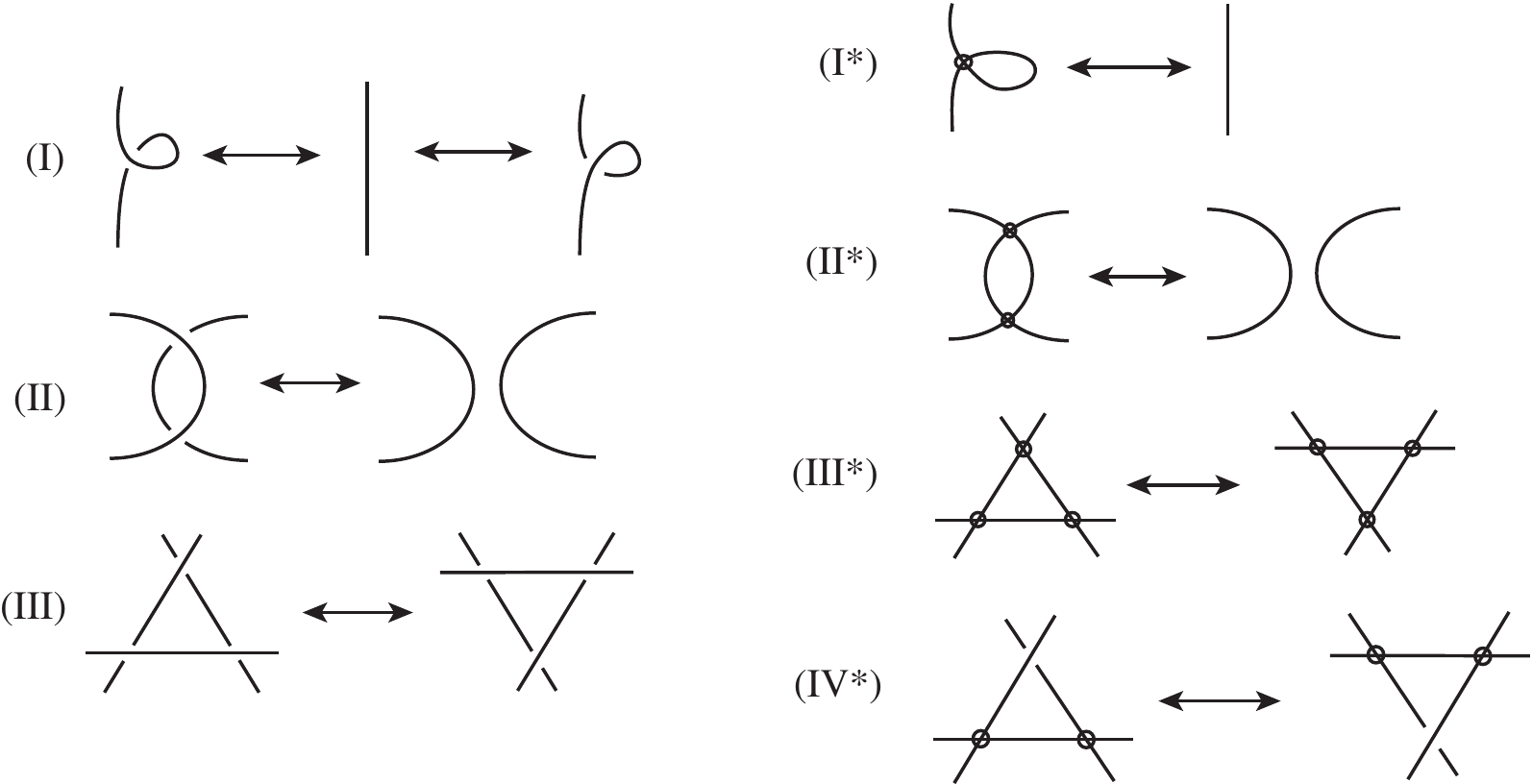}}
\end{center}
\caption{Reidemeister moves for virtual knots}
\label{F:reidemeister}
\end{figure}

As with virtual knots, we will define {\em virtual rational tangles} as equivalence classes of diagrams, modulo the Reidemeister moves shown in Figure \ref{F:reidemeister}.  In this case, the diagrams are those that can be formed (as with classical rational tangles) from the trivial tangle by successively twisting pairs of neighboring endpoints.  However, in addition to twists that produce a classical crossing (positive or negative), we also allow ``virtual twists" that produce a virtual crossing.  (Alternatively, we can simply take the diagrams formed by doing classical twists, and replacing some of the crossings with virtual crossings.)  We also allow, of course, diagrams that are equivalent to these by Reidemeister moves and by isotopies fixing the endpoints of the tangle.

\subsection{Flypes} \label{SS:flypes}

We will focus on a weaker equivalence among virtual knots than that generated by the Reidemeister moves, called {\em F-equivalence}.  Two virtual knots are {\em F-equivalent} if they are equivalent modulo planar isotopy, Reidemeister moves, and classical and virtual {\em flypes}, shown in Figure~\ref{F:flypes}.  In classical knot theory, knots that are equivalent by (classical) flypes are also isotopic, so $F$-equivalence is the same as isotopy.  However, in virtual knot theory this is no longer true. Kauffman \cite{ka2} provides an example of a non-trivial virtual knot that is $F$-equivalent to the unknot. 


\begin{figure}[htbp]
\begin{center}
\begin{align*}
{\rm Classical\ Flype:} &\qquad\scalebox{.6}{\raisebox{-.3\height}{\includegraphics{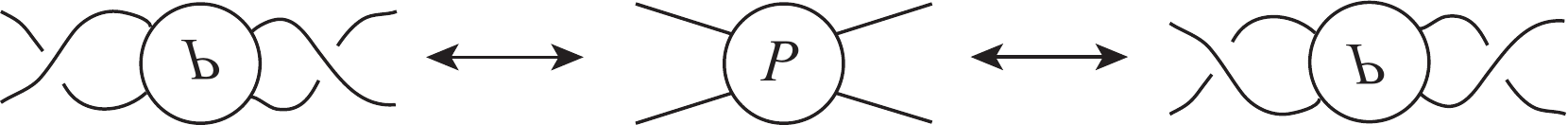}}}\\\\
{\rm Virtual\ Flype:} &\qquad\scalebox{.6}{\raisebox{-.3\height}{\includegraphics{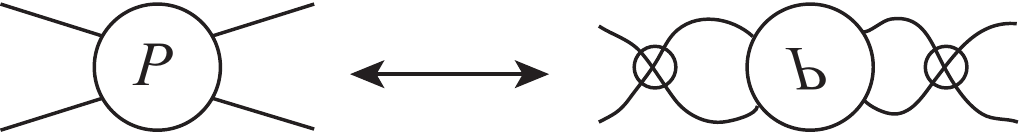}}}
\end{align*}
\end{center}
\caption{Classical and virtual flypes}
\label{F:flypes}
\end{figure}

Recall from Section \ref{S:tangles} that every rational tangle is isotopic to a {\em basic tangle} that can be described by a vector $(a_1, a_2, \dots, a_n)$.  This isotopy is through a series of flypes \cite{gk}.  Similarly, any virtual rational tangle is $F$-equivalent to a (virtual) basic tangle. Moreover, in an elementary horizontal or vertical tangle of classical and virtual crossings, the virtual crossings can be moved to the end by virtual flypes.  Since a pair of adjacent virtual crossings will cancel, any elementary tangle is $F$-equivalent to one with either no virtual crossings, or a single virtual crossing at one end.  We will now denote the elementary tangles with $n$ classical crossings by $[n^\e]$ or $\frac{1}{[n^\e]}$, where the tangle is classical if $\e = 0$ and has a single virtual crossing (in addition to $n$ classical crossings) if $\e = 1$. (Note that $[0^1]$ and $[\infty^1]$ both represent the tangle $\scalebox{.15}{\raisebox{-.3\height}{\includegraphics{VC.pdf}}}$.)

As a result, as with classical rational tangles, every virtual rational tangle is $F$-equivalent to a virtual basic tangle $(a_1^{\e_1}, a_2^{\e_2}, \dots, a_n^{\e_n})$ where 
$$(a_1^{\e_1}, a_2^{\e_2}, \dots, a_n^{\e_n}) = \left(\cdots \left(\left([a_1^{\e_1}]*\frac{1}{[a_2^{\e_2}]}\right)+[a_3^{\e_3}]\right)* \cdots * \frac{1}{[a_{n-1}^{\e_{n-1}}]} \right) + [a_n^{\e_n}].$$
Here we allow $[a_i^{\e_i}] = [0^1]$ or $\frac{1}{[a_i^{\e_i}]} = [\infty^1]$, but not $[a_i^{\e_i}] = [0^0]$ or $\frac{1}{[a_i^{\e_i}]} = [\infty^0]$, for $1 < i < n$.  We can assume that $a_n$ is a horizontal tangles if we allow $[a_n^{\e_n}] = [0^0]$ and, we can assume $[a_1^{\e_1}]$ is either a horizontal tangle or $[\infty^0]$. This permits us to use a single expression for a basic tangle, as in the classical case.

\subsection{The bracket polynomial and $F$-equivalence}

The bracket polynomial extends easily to virtual knots and virtual tangles; as before, we apply rule (1) to every classical crossing.  For classical knots and tangles, the bracket polynomial is an invariant of regular isotopy; Kauffman \cite{ka2} showed that, for virtual knots, it is also invariant under the virtual Reidemeister moves. We will extend this to show that for virtual knots and tangles the bracket polynomial is an invariant of {\em regular $F$-equivalence} (i.e. regular isotopy, virtual Reidemeister moves, and flypes).

\begin{thm} \label{T:flypes}
The bracket polynomial for virtual knots and tangles is invariant under classical and virtual flypes (and hence is an invariant of regular $F$-equivalence).
\end{thm}
\begin{proof}
We want to compare the results of applying the bracket polynomials to the various tangles in Figure \ref{F:flypes}.  After we resolve all the classical crossings, we are left with a state sum where each state is a union of loops and tangles with only virtual crossings.  Under the virtual Reidemeister moves, a loop with only virtual crossings is equivalent to the unknot, and a tangle with only virtual crossings is equivalent to one of three basic tangles.  So the bracket of a tangle $P$ becomes a sum of three terms:
$$\left\langle P \right\rangle = f_P(A)\left\langle \scalebox{.15}{\raisebox{-.3\height}{\includegraphics{B3.pdf}}} \right\rangle + g_P(A)\left\langle \scalebox{.15}{\raisebox{-.3\height}{\includegraphics{B2.pdf}}} \right\rangle + h_P(A)\left\langle \scalebox{.15}{\raisebox{-.3\height}{\includegraphics{VC.pdf}}} \right\rangle$$
Since all three basic tangles are symmetric when flipped upside down, the result of applying the bracket after a virtual flype (as shown in Figure \ref{F:flypes}) is
$$f_P(A)\left\langle \scalebox{.15}{\raisebox{-.3\height}{\includegraphics{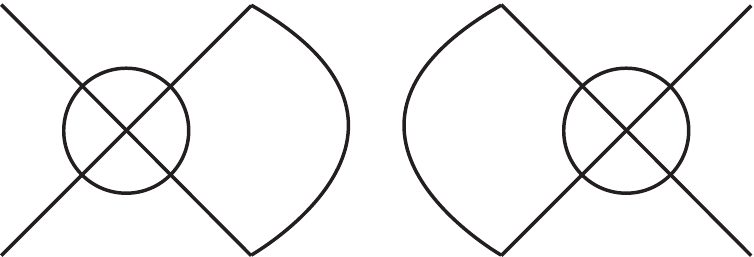}}} \right\rangle + g_P(A)\left\langle \scalebox{.15}{\raisebox{-.3\height}{\includegraphics{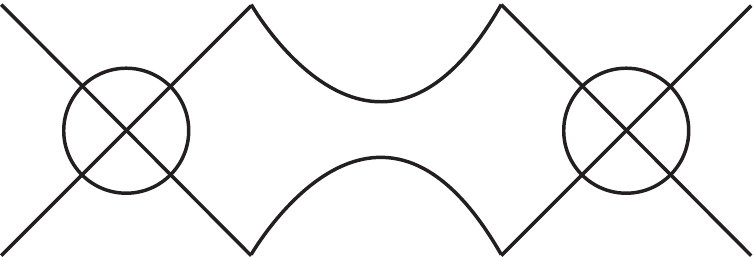}}} \right\rangle + h_P(A)\left\langle \scalebox{.15}{\raisebox{-.3\height}{\includegraphics{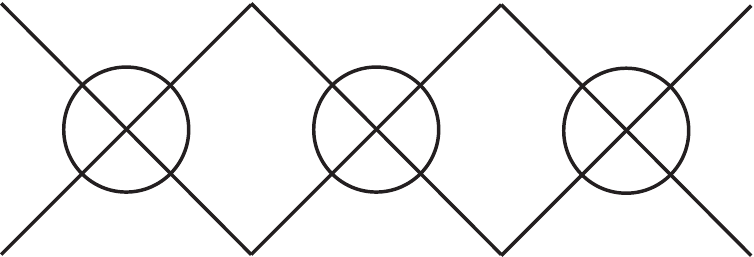}}} \right\rangle$$
By virtual Reidemeister moves I* and II*, this is equal to $\left\langle P \right\rangle$.

Now we consider a classical flype (with a possibly virtual tangle $P$).  We will consider one of the classical flypes shown in Figure \ref{F:flypes}; the other is similar.  In this case, the value of the bracket polynomial of the flyped tangle is
$$f_P(A)\left\langle \scalebox{.15}{\raisebox{-.3\height}{\includegraphics{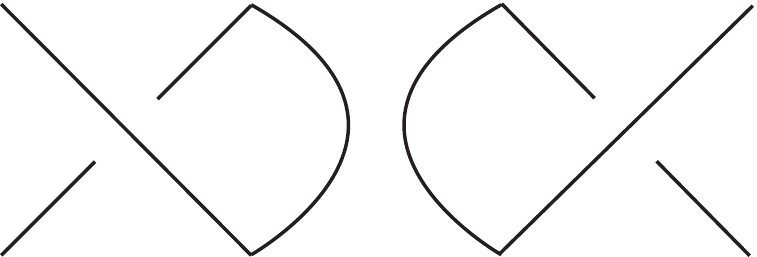}}} \right\rangle + g_P(A)\left\langle \scalebox{.15}{\raisebox{-.3\height}{\includegraphics{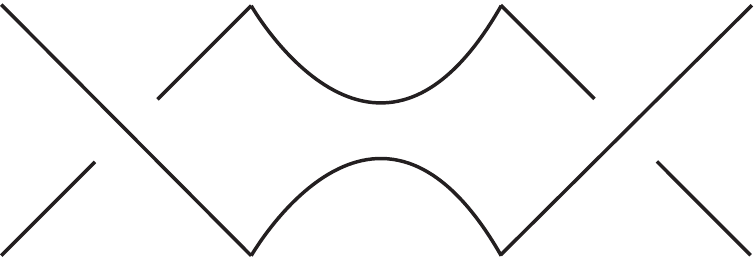}}} \right\rangle + h_P(A)\left\langle \scalebox{.15}{\raisebox{-.3\height}{\includegraphics{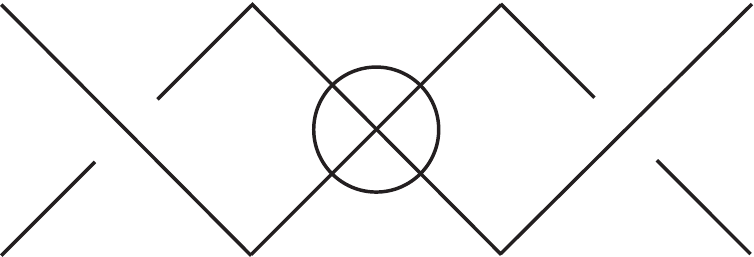}}} \right\rangle$$
We consider each of the three terms in turn:
\begin{align*}
\left\langle \scalebox{.15}{\raisebox{-.3\height}{\includegraphics{C3.pdf}}} \right\rangle &= (-A^3)(-A^{-3})\left\langle \scalebox{.15}{\raisebox{-.3\height}{\includegraphics{B3.pdf}}} \right\rangle = \left\langle \scalebox{.15}{\raisebox{-.3\height}{\includegraphics{B3.pdf}}} \right\rangle \\
\left\langle \scalebox{.15}{\raisebox{-.3\height}{\includegraphics{C2.pdf}}} \right\rangle &= \left\langle \scalebox{.15}{\raisebox{-.3\height}{\includegraphics{B2.pdf}}} \right\rangle\ {\rm by\ Reidemeister\ II}\\
\left\langle \scalebox{.15}{\raisebox{-.3\height}{\includegraphics{VCflype.pdf}}} \right\rangle &= A\left\langle \scalebox{.15}{\raisebox{-.3\height}{\includegraphics{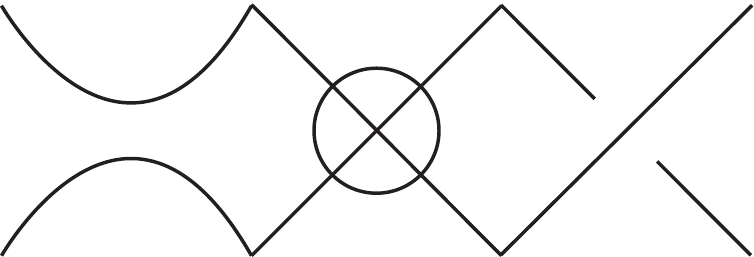}}} \right\rangle + A^{-1}\left\langle \scalebox{.15}{\raisebox{-.3\height}{\includegraphics{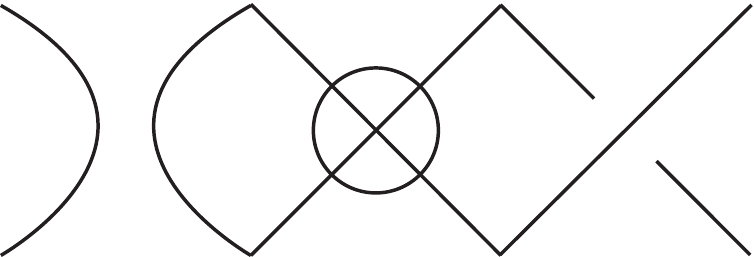}}} \right\rangle \\
&= A^2\left\langle \scalebox{.15}{\raisebox{-.3\height}{\includegraphics{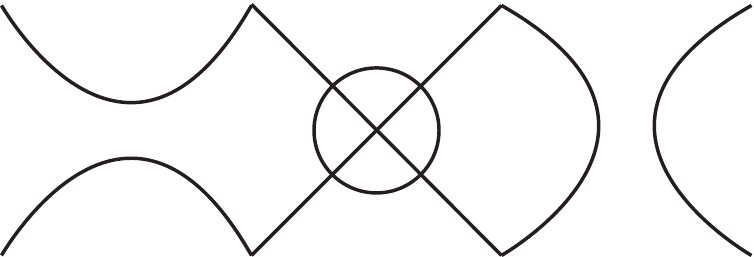}}} \right\rangle + \left\langle \scalebox{.15}{\raisebox{-.3\height}{\includegraphics{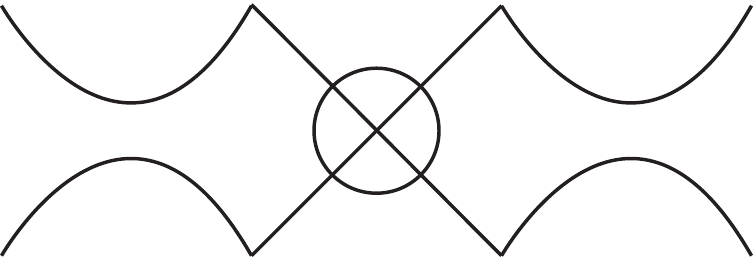}}} \right\rangle + \left\langle \scalebox{.15}{\raisebox{-.3\height}{\includegraphics{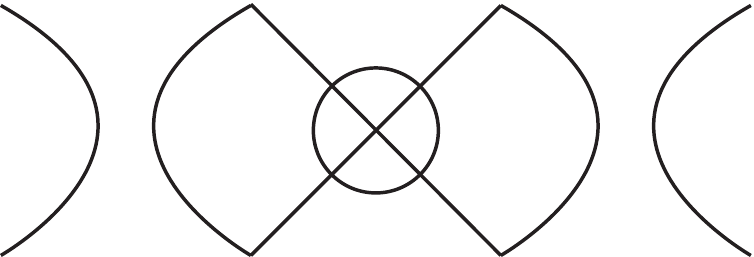}}} \right\rangle + A^{-2}\left\langle \scalebox{.15}{\raisebox{-.3\height}{\includegraphics{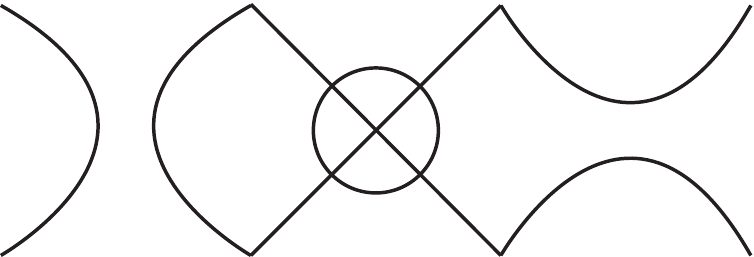}}} \right\rangle\\
&= A^2\left\langle \scalebox{.15}{\raisebox{-.3\height}{\includegraphics{B3.pdf}}} \right\rangle + \left\langle \scalebox{.15}{\raisebox{-.3\height}{\includegraphics{VC.pdf}}} \right\rangle + (-A^2-A^{-2})\left\langle \scalebox{.15}{\raisebox{-.3\height}{\includegraphics{B3.pdf}}} \right\rangle + A^{-2}\left\langle \scalebox{.15}{\raisebox{-.3\height}{\includegraphics{B3.pdf}}} \right\rangle\\
&= \left\langle \scalebox{.15}{\raisebox{-.3\height}{\includegraphics{VC.pdf}}} \right\rangle
\end{align*}
So, once again, the bracket polynomial of the flyped tangle is equal to $\left\langle P \right\rangle$.
\end{proof}

\noindent{\sc Remark:} Kauffman \cite{ka2} showed that ``virtualizing" a crossing did not change the bracket polynomial:
$$\left\langle \scalebox{.25}{\raisebox{-.3\height}{\includegraphics{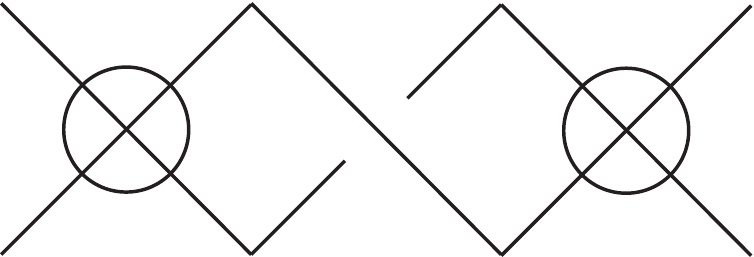}}} \right\rangle = \left\langle \scalebox{.25}{\raisebox{-.3\height}{\includegraphics{B1.pdf}}} \right\rangle$$
We can now see that this is a special case of Theorem \ref{T:flypes}, where the tangle being flyped consists of a single crossing. Kamada \cite{kam} calls this virtualization move a Kauffman flype, or {\em K-flype}, and defines two virtual knots as {\em K-equivalent} if they are equivalent modulo the extended Reidemeister moves and $K$-flypes. Clearly $K$-equivalence implies $F$-equivalence, but it is an open question whether the converse is true.

\section{Fraction of a virtual rational tangle}

\noindent As we observed in the proof of Theorem \ref{T:flypes} the bracket of a virtual tangle reduces to a sum of three terms:
$$\left\langle T \right\rangle = f_T(A)\left\langle \scalebox{.15}{\raisebox{-.3\height}{\includegraphics{B3.pdf}}} \right\rangle + g_T(A)\left\langle \scalebox{.15}{\raisebox{-.3\height}{\includegraphics{B2.pdf}}} \right\rangle + h_T(A)\left\langle \scalebox{.15}{\raisebox{-.3\height}{\includegraphics{VC.pdf}}} \right\rangle$$

\medskip

\noindent The bracket for tangles is an invariant of regular $F$-equivalence, but not an isotopy invariant, since it is not invariant under Reidemeister move I.  As in the classical case, to obtain an isotopy invariant we define a ratio; but now we want one that uses all three coefficients:
$$R_T(A) = \dfrac{f_T(A) + h_T(A)}{g_T(A) + h_T(A)}$$
As before, we define the {\em conductance} of the tangle by $C(T) = iR_T(\sqrt{i})$.  Like the bracket, $R_T$ and $C(T)$ are invariants of regular $F$-equivalence, but they are also isotopy invariants.  In the case of a classical rational tangle, $h_T(A) = 0$, and this agrees with the fraction of the tangle.  An important property of the classical conductance invariant is that it is additive \cite{gk}.  In the virtual case, this is no longer true in general, but it still holds in many important cases.

\begin{lem} \label{L:additivity}
Suppose $T$ and $S$ are tangles and that 
$$\left\langle T \right\rangle = a\left\langle \scalebox{.15}{\raisebox{-.3\height}{\includegraphics{B3.pdf}}} \right\rangle + b\left\langle \scalebox{.15}{\raisebox{-.3\height}{\includegraphics{B2.pdf}}} \right\rangle + c\left\langle \scalebox{.15}{\raisebox{-.3\height}{\includegraphics{VC.pdf}}} \right\rangle$$
$$\left\langle S \right\rangle = d\left\langle \scalebox{.15}{\raisebox{-.3\height}{\includegraphics{B3.pdf}}} \right\rangle + e\left\langle \scalebox{.15}{\raisebox{-.3\height}{\includegraphics{B2.pdf}}} \right\rangle + f\left\langle \scalebox{.15}{\raisebox{-.3\height}{\includegraphics{VC.pdf}}} \right\rangle$$
where the coefficients $a, b, c, d, e, f$ are polynomials in $A$.  If {\em either} $c = 0$ or $f = 0$ (in particular, if either $T$ or $S$ is a classical tangle), then $C(T+S) = C(T) + C(S)$ and $C(T*S) = \dfrac{1}{\frac{1}{C(T)} + \frac{1}{C(S)}}$.
\end{lem}
\begin{proof}
\begin{align*}
\left\langle T+S \right\rangle &= ad\left\langle \i + \i \right\rangle + ae\left\langle \i + \z \right\rangle + af\left\langle \i + \v \right\rangle\\
&\qquad+ bd\left\langle \z + \i \right\rangle + be\left\langle \z + \z \right\rangle + bf\left\langle \z + \v \right\rangle\\
&\qquad+ cd\left\langle \v + \i \right\rangle + ce\left\langle \v + \z \right\rangle + cf\left\langle \v + \v \right\rangle\\
&= ad\left\langle O \cup \i \right\rangle + ae\left\langle \i \right\rangle + af\left\langle \i \right\rangle\\
&\qquad+ bd\left\langle \i \right\rangle + be\left\langle \z \right\rangle + bf\left\langle \v \right\rangle\\
&\qquad+ cd\left\langle \i \right\rangle + ce\left\langle \v  \right\rangle + cf\left\langle \z \right\rangle\\
&= (ad(-A^2-A^{-2}) + ae + af + bd + cd) \left\langle \i \right\rangle + (be + cf) \left\langle \z \right\rangle + (bf + ce) \left\langle \v \right\rangle
\end{align*}
If we let $A = \sqrt{i}$, then $-A^2 - A^{-2} = 0$, so 
$$C(T+S) = i \frac{ae + af + bd + cd + bf + ce}{be+cf +bf + ce} = i \frac{(a+c)(e+f) + (b+c)(d+f) - 2cf}{(b+c)(e+f)}$$
$$= i\left(\frac{a+c}{b+c} + \frac{d+f}{e+f} - \frac{2cf}{(b+c)(e+f)}\right) = C(T) + C(S) - \frac{2cfi}{(b+c)(e+f)}$$
So if either $c= 0$ or $f = 0$, then $C(T+S) = C(T) + C(S)$.  A similar calculation shows that if $c=0$ or $f=0$, then $C(T*S) = \dfrac{1}{\frac{1}{C(T)} + \frac{1}{C(S)}}$.
\end{proof}

The following lemma gives the value of $\langle T\rangle$ when $T$ is an elementary tangle.

\begin{lem} \label{L:elementary}
In the formulas below, $\s(n) = \frac{n}{\abs{n}}$ is the sign of $n$.  Recall that $[n^\e]$ is a classical elementary tangle if $\e = 0$ and has a single virtual crossing at one end if $\e = 1$.
$$\langle [n^0] \rangle = A^{n-2\s(n)}\sum_{k=0}^{\abs{n}-1}{\left(-A^{-4\s(n)}\right)^k} \left\langle \scalebox{.15}{\raisebox{-.3\height}{\includegraphics{B3.pdf}}} \right\rangle + A^n \left\langle \scalebox{.15}{\raisebox{-.3\height}{\includegraphics{B2.pdf}}} \right\rangle$$
$$\langle [n^1] \rangle = A^{n-2\s(n)}\sum_{k=0}^{\abs{n}-1}{\left(-A^{-4\s(n)}\right)^k} \left\langle \scalebox{.15}{\raisebox{-.3\height}{\includegraphics{B3.pdf}}} \right\rangle + A^n \left\langle \scalebox{.15}{\raisebox{-.3\height}{\includegraphics{VC.pdf}}} \right\rangle$$
$$\left\langle \frac{1}{[n^0]} \right\rangle =  A^{-n} \left\langle \scalebox{.15}{\raisebox{-.3\height}{\includegraphics{B3.pdf}}} \right\rangle + A^{-n+2\s(n)}\sum_{k=0}^{\abs{n}-1}{\left(-A^{4\s(n)}\right)^k} \left\langle \scalebox{.15}{\raisebox{-.3\height}{\includegraphics{B2.pdf}}} \right\rangle$$
$$\left\langle \frac{1}{[n^1]} \right\rangle =  A^{-n} \left\langle \scalebox{.15}{\raisebox{-.3\height}{\includegraphics{B3.pdf}}} \right\rangle + A^{-n+2\s(n)}\sum_{k=0}^{\abs{n}-1}{\left(-A^{4\s(n)}\right)^k} \left\langle \scalebox{.15}{\raisebox{-.3\height}{\includegraphics{VC.pdf}}} \right\rangle$$
\end{lem}
\begin{proof}
The first two formulas are an easy induction.  The second two then follow from the observation that $\frac{1}{[n^\e]}$ is the result of rotating $[-n^\e]$ a quarter turn clockwise.
\end{proof}

Once we have computed the bracket polynomial, it is easy to find $R_T$ and $C(T)$.

\begin{cor} \label{C:RCelem}
$$R_{[n^0]}(A) = A^{-2\s(n)}\sum_{k=0}^{\abs{n}-1}{\left(-A^{-4\s(n)}\right)^k}, \qquad C([n^0]) = iR_{[n^0]}(\sqrt{i}) = n$$
$$R_{[n^1]}(A) = R_{[n^0]}(A) + 1, \qquad C([n^1]) = iR_{[n^1]}(\sqrt{i}) = n + i$$
$$R_{\frac{1}{[n^0]}}(A) = \frac{1}{R_{[-n^0]}(A)}, \qquad C\left(\frac{1}{[n^0]}\right) = \frac{1}{n}$$
$$R_{\frac{1}{[n^1]}}(A) = \frac{1}{R_{[-n^0]}(A) + 1}, \qquad C\left(\frac{1}{[n^1]}\right) = \frac{1}{n-i}$$
\end{cor}

\medskip

There is a useful relationship between $R_T$ for a tangle and for the results of adding a virtual crossing to the right or below the tangle.

\begin{lem} \label{L:ratio}
Let $T$ be a virtual tangle, and let $T'$ and $T''$ be the virtual tangles shown below:
$$T' = \scalebox{.5}{\raisebox{-.3\height}{\includegraphics{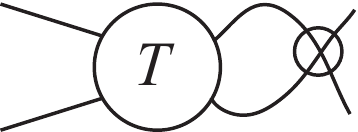}}}, \qquad T'' = \scalebox{.5}{\raisebox{-.3\height}{\includegraphics{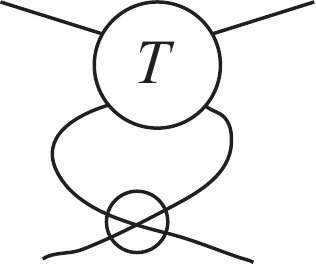}}}$$
Then $R_T = R_{T'}R_{T''}$ and $C(T) = -iC(T')C(T'')$.
\end{lem}
\begin{proof}
Say that $\left\langle T \right\rangle = f_T(A)\left\langle \scalebox{.15}{\raisebox{-.3\height}{\includegraphics{B3.pdf}}} \right\rangle + g_T(A)\left\langle \scalebox{.15}{\raisebox{-.3\height}{\includegraphics{B2.pdf}}} \right\rangle + h_T(A)\left\langle \scalebox{.15}{\raisebox{-.3\height}{\includegraphics{VC.pdf}}} \right\rangle$.  Then we can find the brackets for $T'$ and $T''$ just by adding a virtual crossing on the side or below each of the basic tangles, and applying virtual Reidemeister moves.
$$\left\langle T' \right\rangle = f_T(A)\left\langle \scalebox{.15}{\raisebox{-.3\height}{\includegraphics{B3.pdf}}} \right\rangle + g_T(A)\left\langle \scalebox{.15}{\raisebox{-.3\height}{\includegraphics{VC.pdf}}} \right\rangle + h_T(A)\left\langle \scalebox{.15}{\raisebox{-.3\height}{\includegraphics{B2.pdf}}} \right\rangle$$
$$\left\langle T'' \right\rangle = f_T(A)\left\langle \scalebox{.15}{\raisebox{-.3\height}{\includegraphics{VC.pdf}}} \right\rangle + g_T(A)\left\langle \scalebox{.15}{\raisebox{-.3\height}{\includegraphics{B2.pdf}}} \right\rangle + h_T(A)\left\langle \scalebox{.15}{\raisebox{-.3\height}{\includegraphics{B3.pdf}}} \right\rangle$$
This means that:
$$R_T(A) = \dfrac{f_T(A) + h_T(A)}{g_T(A) + h_T(A)}, \qquad R_{T'}(A) = \dfrac{f_T(A) + g_T(A)}{h_T(A) + g_T(A)}, \qquad R_{T''}(A) = \dfrac{h_T(A) + f_T(A)}{g_T(A) + f_T(A)}$$
So $R_T = R_{T'}R_{T''}$ and (since $i^2 = -1$) $C(T) = -iC(T')C(T'')$.
\end{proof}

We can now provide a recursive formula for computing the conductance of a virtual basic rational tangle.

\begin{lem} \label{recursion}
Consider a tangle $T = (a_1^{\e_1}, a_2^{\e_2}, \dots, a_n^{\e_n})$.  If $a_n$ represents a {\em horizontal} tangle, then
$$C(T) = a_n + \dfrac{C(a_1^{\e_1}, a_2^{\e_2}, \dots, a_{n-1}^{\e_{n-1}})}{(-i C(a_1^{\e_1}, a_2^{\e_2}, \dots, a_{n-2}^{\e_{n-2}}, a_{n-1}^{(1-\e_{n-1})}))^{\e_n}}$$
and if $a_n$ represents a {\em vertical} tangle, then
$$C(T) = \dfrac{1}{a_n + \dfrac{(-i C(a_1^{\e_1}, a_2^{\e_2}, \dots, a_{n-2}^{\e_{n-2}}, a_{n-1}^{(1-\e_{n-1})}))^{\e_n}}{C(a_1^{\e_1}, a_2^{\e_2}, \dots, a_{n-1}^{\e_{n-1}})}}$$
\end{lem}
\begin{proof}
We first consider the case when $a_n^{\e_n}$ represents a horizontal tangle and $\e_n = 0$.  Let $T_{n-1} = (a_1^{\e_1}, a_2^{\e_2}, \dots, a_{n-1}^{\e_{n-1}})$.  Then $T = T_{n-1} + [a_n]$. Since $[a_n]$ is a classical tangle, Lemma \ref{L:additivity} implies $C(T) = C(T_{n-1}) + C([a_n]) = C(T_{n-1}) + a_n$.  So (since $\e_n = 0$)
$$C(T) = a_n + \frac{C(a_1^{\e_1}, a_2^{\e_2}, \dots, a_{n-1}^{\e_{n-1}})}{1} = a_n + \dfrac{C(a_1^{\e_1}, a_2^{\e_2}, \dots, a_{n-1}^{\e_{n-1}})}{(-i C(a_1^{\e_1}, a_2^{\e_2}, \dots, a_{n-2}^{\e_{n-2}}, a_{n-1}^{(1-\e_{n-1})}))^{\e_n}}.$$
On the other hand, if $\e_n = 1$, then $T = T_{n-1} + \v + [a_n]$ (up to $F$-equivalence), so $C(T) = C\left(T_{n-1} + \v\right) + C([a_n]) = C\left(T_{n-1} + \v\right) + a_n$.  By Lemma \ref{L:ratio}, $C\left(T_{n-1} + \v\right) = i\dfrac{C(T_{n-1})}{C\left(T_{n-1}*\v\right)}$.  If $a_n$ represents a horizontal tangle, then $a_{n-1}$ represents a vertical tangle, so $T_{n-1}*\v$ either adds a virtual crossing to $\frac{1}{[a_{n-1}^0]}$ or (by Reidemeister move (II$^*$)) cancels a virtual crossing from $\frac{1}{[a_{n-1}^1]}$.  So 
$$C(T) = a_n + i \dfrac{C(T_{n-1})}{C(a_1^{\e_1}, a_2^{\e_2}, \dots, a_{n-2}^{\e_{n-2}}, a_{n-1}^{(1-\e_{n-1})})} = a_n + \dfrac{C(a_1^{\e_1}, a_2^{\e_2}, \dots, a_{n-1}^{\e_{n-1}})}{(-i C(a_1^{\e_1}, a_2^{\e_2}, \dots, a_{n-2}^{\e_{n-2}}, a_{n-1}^{(1-\e_{n-1})}))^{\e_n}}.$$

\medskip

Now we consider the case when $a_n^{\e_n}$ represents a vertical tangle and $\e_n = 0$.  Then $T = T_{n-1}*\frac{1}{[a_n^0]}$, and Lemma \ref{L:additivity} implies 
$$C(T) = \dfrac{1}{\dfrac{1}{C(T_{n-1})} + \dfrac{1}{C\left(\frac{1}{[a_n^0]}\right)}} = \frac{1}{\dfrac{1}{C(T_{n-1})} + a_n} = \dfrac{1}{a_n + \dfrac{1}{C(a_1^{\e_1}, a_2^{\e_2}, \dots, a_{n-1}^{\e_{n-1}})}}$$
$$= \dfrac{1}{a_n + \dfrac{(i C(a_1^{\e_1}, a_2^{\e_2}, \dots, a_{n-2}^{\e_{n-2}}, a_{n-1}^{(1-\e_{n-1})}))^{\e_n}}{C(a_1^{\e_1}, a_2^{\e_2}, \dots, a_{n-1}^{\e_{n-1}})}}.$$
On the other hand, if $\e_n = 1$, then $T = T_{n-1}*\v*\frac{1}{[a_n]}$, so 
$$C(T) = \frac{1}{\dfrac{1}{C\left(T_{n-1}*\v\right)} + \dfrac{1}{C\left(\frac{1}{[a_n^0]}\right)}} = \frac{1}{\dfrac{1}{C\left(T_{n-1}*\v\right)} + a_n}.$$
By Lemma \ref{L:ratio}, $C\left(T_{n-1} * \v\right) = i\dfrac{C(T_{n-1})}{C\left(T_{n-1} + \v\right)}$.  If $a_n$ represents a vertical tangle, then $a_{n-1}$ represents a horizontal tangle, so $T_{n-1}+\v$ either adds a virtual crossing to $[a_{n-1}^0]$ or cancels one in $[a_{n-1}^1]$. So
$$C(T) = \frac{1}{a_n + \dfrac{C\left(T_{n-1} + \v\right)}{iC(T_{n-1})}} = \dfrac{1}{a_n + \dfrac{(-i C(a_1^{\e_1}, a_2^{\e_2}, \dots, a_{n-2}^{\e_{n-2}}, a_{n-1}^{(1-\e_{n-1})}))^{\e_n}}{C(a_1^{\e_1}, a_2^{\e_2}, \dots, a_{n-1}^{\e_{n-1}})}}.$$
\end{proof}

From Lemma \ref{recursion}, a quick inductive argument allows us to write the conductance as a (generalized) complex continued fraction, generalizing Goldman and Kauffman's result for the conductance of a classical rational tangle.

\begin{thm} \label{T:fraction}
Consider a virtual rational tangle $T = (a_1^{\e_1}, a_2^{\e_2}, \dots, a_n^{\e_n})$.  We may assume $a_1$ and $a_n$ represent horizontal tangles, if we allow $[a_n^{\e_n}] = [0^0]$ and the single exception $[a_1^{\e_1}] = [\infty^0]$, so $n$ is assumed to be odd.  Then the conductance $C(T)$ can be written as a (generalized) complex continued fraction
$$C(T) = a_n + \cfrac{b_{n}}{a_{n-1}+\cfrac{b_{n-1}}{a_{n-2} + \cfrac{b_{n-2}}{\ddots + \cfrac{b_2}{a_1 + b_1}}}}$$
where 
$$b_{k} = \left\{ \begin{matrix} 1, & {\rm if}\ \e_k = 0 \\ -iC(a_1^{\e_1}, a_2^{\e_2}, \dots, a_{k-1}^{1-\e_{k-1}}), & {\rm if}\ \e_k = 1, k\ {\rm even} \\ (-iC(a_1^{\e_1}, a_2^{\e_2}, \dots, a_{k-1}^{1-\e_{k-1}}))^{-1}, & {\rm if}\ \e_k = 1, k\ {\rm odd} \end{matrix} \right.,\   {\rm for}\  2 \leq k \leq n,\ {\rm and}$$
$$b_1 = \left\{ \begin{matrix} 0 & {\rm if}\ e_1 = 0 \\ i & {\rm if}\ e_1 = 1 \end{matrix}\right..$$
\end{thm}


\section{Examples} \label{S:examples}

To illustrate the use of Theorem \ref{T:fraction}, we will compute the conductance for all virtual rational tangles $T = (a_1^{\e_1}, a_2^{\e_2}, \dots, a_n^{\e_n})$ with $n = 1, 2, 3$ and $a_1$ representing a horizontal tangle (or possibly $[a_1^{\e_1}] = [\infty^0]$).  So $a_n$ represents a vertical tangle if $n$ is even and a horizontal tangle if $n$ is odd.  If $n$ is even, we can apply Theorem \ref{T:fraction} to the tangle $T = (a_1^{\e_1}, a_2^{\e_2}, \dots, a_n^{\e_n}, 0^0)$. The conductance of a single horizontal tangle is given (as in Corollary \ref{C:RCelem}) by:
$$C(a^0) = a \qquad \qquad C(a^1) = a + i$$
For tangles with two components, there are four possibilities.  Note that if $[a^{\e_1}] = [\infty^0]$, then $(\infty^0, b^{\e_2})$ represents a single vertical tangle. (In this case, we can use the formulas below by taking the limit as $a \rightarrow \infty$.)
\begin{align*}
C(a^0, b^0) &= \dfrac{1}{b + \dfrac{1}{a}} = \dfrac{a}{ab+1} \\
C(a^1, b^0) &= \dfrac{1}{b + \cfrac{1}{a+i}} = \dfrac{a+b+a^2b}{(ab+1)^2+b^2} + \dfrac{1}{(ab+1)^2+b^2}\,i\\
C(a^0, b^1) &= \dfrac{1}{b + \cfrac{-iC(a^1)}{a}} = \dfrac{1}{b + \cfrac{-i(a+i)}{a}} = \dfrac{a(ab+1)}{(ab+1)^2 + a^2} + \dfrac{a^2}{(ab+1)^2 + a^2}\,i\\
C(a^1, b^1) &= \dfrac{1}{b + \cfrac{-iC(a^0)}{a+i}} = \dfrac{1}{b + \cfrac{-ia}{a+i}} = \dfrac{-a+b+a^2b}{(ab-1)^2+a^2+b^2-1} + \dfrac{a^2}{(ab-1)^2+a^2+b^2-1}\,i
\end{align*}
Now we turn to tangles with three components.  Since $C(a^{\e_1}, b^{\e_2}, c^0) = C(a^{\e_1}, b^{\e_2}) + c$ by Lemma \ref{L:additivity}, we we will just compute $C(a^{\e_1}, b^{\e_2}, c^1)$.
\begin{align*}
C(a^0, b^0, c^1) &= c + \cfrac{\left(\dfrac{1}{-iC(a^0, b^1)}\right)}{b + \dfrac{1}{a}} = c + \cfrac{i\left(b + \cfrac{-i(a+i)}{a}\right)}{b + \dfrac{1}{a}} = c + \dfrac{a}{ab+1} + i\\
C(a^1, b^0, c^1) &= c + \cfrac{\left(\dfrac{1}{-iC(a^1, b^1)}\right)}{b + \dfrac{1}{a+i}} = c + \cfrac{i\left(b + \cfrac{-ia}{a+i}\right)}{b + \dfrac{1}{a+i}} = c + \dfrac{a-b+a^2b}{(ab+1)^2 + b^2} + \dfrac{(1+a^2)b^2}{(ab+1)^2 + b^2}\,i \\
C(a^0, b^1, c^1) &= c + \cfrac{\left(\dfrac{1}{-iC(a^0, b^0)}\right)}{b + \cfrac{-iC(a^1)}{a}} = c + \cfrac{i\left(b + \dfrac{1}{a}\right)}{b + \cfrac{-i(a+i)}{a}} = c - \dfrac{a(ab+1)}{(ab+1)^2 + a^2} + \dfrac{(ab+1)^2}{(ab+1)^2 + a^2}\,i \\
C(a^1, b^1, c^1) &= c + \cfrac{\left(\dfrac{1}{-iC(a^1, b^0)}\right)}{b + \cfrac{-iC(a^0)}{a+i}} = c + \cfrac{i\left(b + \cfrac{1}{a+i}\right)}{b + \cfrac{-ia}{a+i}} \\
&= c - \dfrac{a-b+a^2b}{(ab-1)^2 + a^2 + b^2 - 1} - \dfrac{(1+a^2) b^2}{(ab-1)^2 + a^2 + b^2 - 1}\,i
\end{align*}

\section{Questions}

In the case of classical rational tangles, the fraction of the tangle (i.e. the conductance) is a complete invariant, so two rational tangles are isotopic if and only if they have the same fraction.  In the case of virtual rational tangles, the conductance is only an invariant of $F$-equivalence, so we cannot hope for a complete isotopy invariant.  However, we conjecture that it is a complete invariant for $F$-equivalence.

\begin{conj}
The conductance $C(T)$ completely classifies virtual rational tangles up to $F$-equivalence.
\end{conj}

As a special case, we conjecture that the conductance at least distinguishes virtual rational tangles from classical rational tangles.

\begin{conj}
Suppose $T = (a_1^{\e_1}, \dots, a_n^{\e_n})$.  Then $T$ is equivalent to a classical rational tangle if and only if $C(T)$ is a real number.
\end{conj}

We can get some evidence for this conjecture from the examples in Section \ref{S:examples}, and looking at when the conductances are purely real.  In every case, the tangle is $F$-equivalent to a classical rational tangle.  For example, $C(a^1,b^1,c^1)$ is real when $b = 0$, but $(a^1, 0^1, c^1)$ is $F$-equivalent to $(\infty^0, (-a)^0, c^0)$, as shown in Figure \ref{F:(a',0',c')}.

\begin{figure}[htbp]
\begin{center}
\scalebox{.9}{\includegraphics{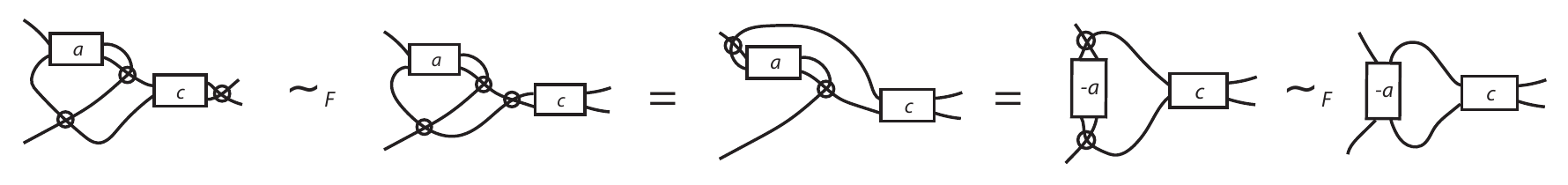}}
\end{center}
\caption{$(a^1, 0^1, c^1) \sim_F (\infty^0, (-a)^0, c^0)$}
\label{F:(a',0',c')}
\end{figure}

This also demonstrates that the same virtual rational tangle may be expressed in different ways.  The proofs that the fraction of a classical rational tangle is a complete invariant given in \cite{gk} and \cite{kl} rely on finding a unique standard form for a rational tangle, and relating these to facts about representations of rational numbers by continued fractions.  In the case of virtual rational tangles, we do not yet have such a standard form, and the fractions we have constructed are not the usual continued fractions, so these facts may no longer hold.

\begin{quest}
Is there a ``standard form" for virtual rational tangles, or for complex generalized continued fractions, that can be used to clarify the extent to which two virtual rational tangles with the same conductance are $F$-equivalent?  
\end{quest}

\end{document}